\documentclass[a4paper,11pt,leqno]{smfart}

\usepackage{amssymb,amsmath,enumerate,verbatim,mathrsfs,graphics,graphicx,mathptm,float,mathtools}
\usepackage[T1]{fontenc}
\usepackage[latin1]{inputenc}
\usepackage[english,francais]{babel}
\usepackage[all]{xy}
\usepackage[pdfstartview=FitH,
            colorlinks=true,
            linkcolor=blue,
            urlcolor=blue,
            citecolor=red,
            bookmarks,
            bookmarksopen=true,
            bookmarksnumbered=true]{hyperref}

\usepackage{cleveref}
\theoremstyle{plain}
\newtheorem{thmsec}{Theorem}[section]
\newtheorem{thm}[thmsec]{Theorem}
\newtheorem{pro}[thmsec]{Proposition}
\newtheorem{lem}[thmsec]{Lemma}
\newtheorem{cor}[thmsec]{Corollary}

\theoremstyle{definition}
\newtheorem{defin}[thmsec]{Definition}

\theoremstyle{remark}
\newtheorem{rem}[thmsec]{Remark}

\newtheorem{eg}[thmsec]{Example}

\def\og{\leavevmode\raise.3ex\hbox{$\scriptscriptstyle\langle\!\langle$~}}
\def\fg{\leavevmode\raise.3ex\hbox{~$\!\scriptscriptstyle\,\rangle\!\rangle$}}

\addtocounter{section}{0}             % Start with section 1
\numberwithin{equation}{section}       % Number formulas within sections

\newcommand{\N}{\mathbb{N}}

\newcommand{\C}{\mathbb{C}}

\newcommand{\pp}{\mathbb{P}^{2}_{\mathbb{C}}}
\newcommand{\pd}{\mathbb{\check{P}}^{2}_{\mathbb{C}}}
\newcommand\Sing{\mathrm{Sing}}
\newcommand\Tang{\mathrm{Tang}}
\newcommand\Leg{\mathrm{Leg}}
\newcommand\IF{\mathrm{I}_{\mathcal{F}}}
\newcommand\radF{\Sigma_{\mathcal{F}}^{\mathrm{rad}}}
\newcommand\radFd{\check{\Sigma}_{\mathcal{F}}^{\mathrm{rad}}}

\newcommand\F{\mathcal{F}}

\newcommand\W{\mathcal{W}}
\newcommand\G{\mathcal{G}}

\newcommand\Omegahesse{\Omega_{\scalebox{0.64}{\ensuremath H}}^{4}}

\newcommand\Omegahessesept{\Omega_{\scalebox{0.64}{\ensuremath H}}^{7}}
\newcommand\Hesse{\mathcal{F}_{\hspace{-0.4mm}\raisebox{-0.2mm}{\tiny{$H$}}}^{4}}

\newcommand\Hessesept{\mathcal{F}_{\hspace{-0.4mm}\raisebox{-0.2mm}{\tiny{$H$}}}^{7}}

\begin{document}
\selectlanguage{english}
\title[Holomorphy of the curvature of planar webs along an invariant curve]{Holomorphy of the curvature of planar webs along an invariant curve and applications to the flatness of certain webs on $\mathbb{P}^{2}_{\mathbb{C}}$}
\date{\today}

\author{Samir \textsc{Bedrouni}}

\address{Facult\'e de Math\'ematiques, USTHB, BP $32$, El-Alia, $16111$ Bab-Ezzouar, Alger, Alg\'erie}
\email{sbedrouni@usthb.dz}


\maketitle{}

\begin{abstract}\selectlanguage{english}
Let $\mathcal{W}=\mathcal{W}_{n}\boxtimes\mathcal{W}_{d-n}$ be a $d$-web on $(\mathbb{C}^2,0)$, where $\mathcal{W}_{n}$ is an $n$-web admitting an irreducible totally invariant~curve~$C$, and $\mathcal{W}_{d-n}$ is a regular $(d-n)$-web transverse to $C$. We show that the curvature of $\mathcal{W}$ is holomorphic along~$C$ if and only if the curvature of $\mathcal{W}_{n}$ is. For each irreducible subweb $\mathcal{W}_{\nu}$ of $\mathcal{W}_{n}$, we distinguish two cases depending on whether $C$ is strongly or weakly invariant by $\mathcal{W}_{\nu}$; in the latter case, we introduce the \textsc{Puiseux} index and exponent $\mathfrak{i}(\mathcal{W}_{\nu},C)$ and $\rho(\mathcal{W}_{\nu},C).$ These invariants allow us to define the notion of non-degeneracy of $\mathcal{W}_{n}$ along~$C.$ When $\mathcal{W}_{n}$ is non-degenerate along $C$, we prove that $K(\mathcal{W}_{n})$, and hence $K(\mathcal{W})$, is holomorphic along $C.$ Moreover, we obtain a natural decomposition of $\mathcal{W}_{n}$ as the product of two subwebs $\mathcal{W}_{n}=\mathcal{W}_{n}^{\rm{str}}\boxtimes\mathcal{W}_{n}^{\rm{wk}}.$ Under the assumption that $\mathcal{W}_{n}^{\rm{wk}}$ is non-degenerate along $C$, we show that the holomorphy of $K(\mathcal{W})$ on $C$ is equivalent to that of $K(\mathcal{W}_{n}^{\rm{str}}).$ As an application, we establish a flatness result for products of dual webs of reduced convex foliations on $\mathbb{P}^{2}_{\mathbb{C}}$ and provide a counterexample to a previous flatness statement by \textsc{Pracias} and \textsc{Falla Luza}.

\end{abstract}

\section*{Introduction}

\noindent A (regular) $d$-web $\W$ on $(\mathbb{C}^2,0)$ is the data of a family $\{\F_1,\F_2,\ldots,\F_d\}$ of regular holomorphic foliations on~$(\mathbb{C}^2,0)$ which are pairwise transverse at the origin. One of the first results in web geometry, due to \textsc{Blaschke}-\textsc{Dubourdieu} \cite{BD28}, characterizes the local equivalence of a germ of a regular $3$-web $\W$ on $(\C^2,0)$ with the trivial $3$-web defined by $\mathrm{d}x.\mathrm{d}y.\mathrm{d}(x+y)$ by the vanishing of a differential $2$-form $K(\W)$, called the \textsc{Blaschke} curvature of $\W$. The curvature of a $d$-web $\W$ with~$d>3$ is defined as the sum of the \textsc{Blaschke} curvatures of all its $3$-subwebs. A web with zero curvature is called flat. This notion is useful for the classification of webs of maximal rank; by \textsc{Mih\u{a}ileanu}'s criterion \cite[Theorem~6.3.4]{PP15}, flatness is a necessary condition for the maximality of the rank.
\smallskip

\noindent A (global) $d$-web $\W$ on a complex surface $M$ is locally given, away from a hypersurface called the discriminant $\Delta(\W)$ of $\W$, by $d$ pairwise transverse holomorphic foliations. In this setting, the curvature $K(\W)$ is a meromorphic $2$-form with poles along $\Delta(\W)$, \emph{cf.}~\S\ref{subsec:courbure-platitude}. When $M$ is the complex projective plane $\pp$, the flatness of a web $\W$ is characterized by the holomorphy of its curvature $K(\W)$ along the generic points of $\Delta(\W).$ This characterization has led to the establishment of several criteria for the holomorphy of the curvature of planar webs, \emph{see} for instance~\cite{PP08,MP13,BM24}. In this paper, we consider webs on $(\mathbb{C}^2,0)$ admitting an~irreducible invariant curve and study the holomorphy of their curvature along such a curve. We then show how to apply our local results to the study of the flatness of certain webs on $\pp$ whose discriminant consists only of invariant curves.
\smallskip

\noindent First, we focus on irreducible $n$-webs $\W_n$ on $(\C^2,0)$ having an irreducible invariant curve $C$ and establish some of their local properties. The irreducibility of such a web $\W_n$ implies that its monodromy group is cyclic of order $n$. This gives rise, in a neighborhood $U$ of a generic point of $C$, to a cyclic ramified \textsc{Galois} covering $\pi\colon\tilde{U}\to U$ of degree $n$ such that the pull-back $\pi^{*}\W_n$ is completely decomposable. Furthermore, we have the following dichotomy: either the lifted curve $\tilde{C}=\pi^{-1}(C)$ is totally invariant by $\pi^{*}\W_n$, in which case $C$ is said to be \textsl{strongly invariant by} $\W_n$; or $\pi^{*}\W_n$ is transverse to $\tilde{C}$, in which case $C$ is said to be \textsl{weakly invariant by} $\W_n$. In the latter case, using the \textsc{Puiseux} parametrizations of the slopes of the leaves of $\W_n$ in a neighborhood of a generic point of $C$, we associate to the pair $(\W_n,C)$ two local invariants: the \textsc{Puiseux} index $\mathfrak{i}(\W_n,C)\in\{1,\ldots,n-1\}$ and the \textsc{Puiseux} exponent $\rho(\W_n,C)=\frac{\mathfrak{i}(\W_n,C)}{n}$, \emph{see}~\S\ref{subsec:fort-faible-invariance}.
\smallskip

\noindent The index $\kappa:=\mathfrak{i}(\W_n,C)$ plays a role in determining the multiplicity of the discriminant $\Delta(\W_n)$ along $C$: we~prove that for $n\geq2$, $$\mathrm{mult}\left(\Delta(\W_{n}),C\right)\geq\kappa(n-1),$$ with equality if and only if $\gcd(n,\kappa)=1.$ In particular, $\mathrm{mult}\left(\Delta(\W_{n}),C\right)$ is minimal, equal to $n-1$, if and only if $\kappa=1$ (\emph{cf.} Lemma~\ref{lem:multiplicite-Delta-W-nu-C-irreductible}).
\smallskip

\noindent Next, we consider a $d$-web on $(\C^2,0)$ of the form $\W=\W_n\boxtimes\W_{d-n}$, where $\W_n$ is an $n$-web (not necessarily irreducible) having an irreducible totally invariant curve~$C$, and $\W_{d-n}$ is a regular $(d-n)$-web transverse to $C$. We show that the curvature of $\W$ is holomorphic along $C$ if and only if the curvature of $\W_n$ is (Theorem~\ref{thm:W-n-W-d-n}).
\smallskip
\vspace{-4.5mm}

\noindent Furthermore, we introduce a particular class of webs, called \textsl{non-degenerate webs along $C$}. The $n$-web $\W_n$ is non-degenerate along $C$ if all its irreducible subwebs admit $C$ as a weakly invariant curve and if in addition their Puiseux indices and exponents satisfy the arithmetic conditions of Definition~\ref{defin:W-nu-non-degenere-le-long-C}. We prove that if $\W_n$ is non-degenerate along $C$, then its curvature $K(\W_n)$ is holomorphic along~$C$ (Proposition~\ref{pro:W-nu-non-degenere-le-long-C}); the same holds for~$K(\W)$ (Theorem~\ref{thm:W-n-W-d-n}).
\smallskip

\noindent From Theorem~\ref{thm:W-n-W-d-n} and Proposition~\ref{pro:W-nu-non-degenere-le-long-C}, we deduce that if $\W_n$ is irreducible and $\mathrm{mult}\left(\Delta(\W_n),C\right)<3(n-1),$ then $K(\W)$ is holomorphic along $C$ (Corollary~\ref{cor:W-nu-W-tr-reg}). If~$n$ is prime or equal to $4$, the condition $\mathrm{mult}\left(\Delta(\W_n),C\right)<3(n-1)$ can be weakened to $\mathrm{mult}\left(\Delta(\W_n),C\right)<n(n-1)$ (Corollary~\ref{cor:W-nu-W-tr-reg-nu-premier-4}).
\smallskip

\noindent Moreover, the $n$-web $\W_n$ naturally decomposes as $\W_n=\W_{n}^{\rm{str}}\boxtimes\W_{n}^{\rm{wk}},$ where $\W_{n}^{\rm{str}}$ (resp. $\W_{n}^{\rm{wk}}$) is the web consisting of the irreducible subwebs of $\W_n$ having $C$ as a strongly (resp. weakly) invariant curve. Under the assumption that $\W_{n}^{\rm{wk}}$ is non-degenerate along $C$, we show that the holomorphy of $K(\W)$ along $C$ is equivalent to that of $K(\W_{n}^{\rm{str}})$ (Theorem~\ref{thm:W-n-fort}).
\smallskip

\noindent Finally, we pass to the framework of global webs on $\pp$. Recall that to every holomorphic foliation $\F$ of degree $d$ on $\pp$ is associated a $d$-web on the dual projective plane $\pd$, called the {\sl \textsc{Legendre} transform} of $\F$ and denoted by $\Leg\F.$ At a generic point $\check{\ell}\in\pd$, corresponding to a line $\ell\subset\pp$, the $d$ directions of $\Leg\F$ are those of the lines dual to the $d$ points where $\F$ is tangent to $\ell.$ Following~\cite{MP13}, a foliation $\F$ on $\pp$ is said to be {\sl convex} if its leaves, other than lines, have no inflection points. It is said to be {\sl reduced convex} if it is convex and has the maximum number of distinct invariant lines, namely $3\cdot\deg\F.$

\noindent Let $\F$ be a reduced convex foliation of degree $d$ on $\pp.$ Theorem~2 of \cite{MP13} states that the $d$-web $\Leg\F$ is flat. Note that the irreducible components of the discriminant $\Delta(\Leg\F)$ are the lines dual to the radial singularities of $\F$ and are therefore all invariant by $\Leg\F$, \emph{cf.}~\cite[\S4.2]{MP13}. Using Definition~\ref{defin:W-nu-non-degenere-le-long-C}, we deduce from~\cite[\S4.2]{MP13} a local decomposition of $\Leg\F$ as a product of two webs: near the line $\check{s}$ dual to a radial singularity $s$ of $\F$, we~have~$\Leg\F=\W_n\boxtimes\W_{d-n}$, where $\W_n$ is a non-degenerate web along $\check{s}$, and $\W_{d-n}$ is regular and transverse~to~$\check{s}.$ Proposition~\ref{pro:W-nu-non-degenere-le-long-C} and Theorem~\ref{thm:W-n-W-d-n} allow us to conclude that the curvature of $\Leg\F$ is holomorphic along~$\check{s}$ and, consequently, to recover the fact that $\Leg\F$ is flat.
\smallskip

\noindent In Section~\S\ref{sec:platitude-tissus-pp}, we generalize this result (\cite[Theorem~2]{MP13}) to the case of a product web $\W=\Leg\F_1\boxtimes\cdots\boxtimes\Leg\F_n,$ where $\F_1,\ldots,\F_n$ are reduced convex foliations on $\pp.$ To do this, we make the following hypotheses:
\begin{itemize}
\smallskip
\item [(\texttt{H1})] {\sl For all $1\leq i<j\leq n$, $\Tang(\F_i,\F_j)$ is a union of invariant lines by $\F_i$ and $\F_j.$}
\smallskip

\item [(\texttt{H2})] {\sl For all $1\leq i<j<k\leq n$, the foliations $\F_i$, $\F_j$, and $\F_k$ have no common singularity that is non-radial for each of them, or radial for each of them with the same radiality order.}
\end{itemize}
\smallskip

\noindent We establish the following result (Theorem~\ref{thm:Leg-ell1-ellr-F1-Fn}): {\sl under hypotheses} (\texttt{H1}) {\sl and} (\texttt{H2}), {\sl the web $\W=\Leg\F_1\boxtimes\cdots\boxtimes\Leg\F_n$ is flat.}
\smallskip

\noindent Hypothesis~(\texttt{H1}) was already proposed by \textsc{Pracias} and \textsc{Falla Luza} in~\cite{PF24}. It ensures that the discriminant $\Delta(\W)$ consists of lines dual to certain singularities of the $\F_i$ and is therefore invariant by~$\W$ (\emph{cf.}~\cite[Chapter~5]{Pra25}). Example~\ref{eg:H2-necessaire} shows that hypothesis (\texttt{H2}) is essential in Theorem~\ref{thm:Leg-ell1-ellr-F1-Fn}. It also provides a counterexample to~\cite[Theorem~A]{PF24}: in our setting, according to this statement, hypothesis~(\texttt{H1}) alone would suffice for $\W$ to be flat.

%==================================================================================================================================================================

\section{Preliminaries}

\subsection{Webs}

Let $d\geq1$ be an integer. A {\sl (global) $d$-web} $\W$ on a complex surface $M$ is given by an open covering $(U_{i})_{i\in I}$ of $M$ and a collection of $d$-symmetric $1$-forms $\omega_{i}\in \mathrm{Sym}^{d}\Omega^{1}_{M}(U_{i})$, with isolated zeros, satisfying:
\begin{itemize}
\item [\textit{(1)}] there exists $g_{ij}\in \mathcal{O}^{*}_{M}(U_{i}\cap U_{j})$ such that $\omega_i$ coincides with $g_{ij}\omega_j$ on $U_i\cap U_j$;

\item [\textit{(2)}] for every generic point $m$ of $U_i$, $\omega_i(m)$ factors as the product of $d$ pairwise linearly independent\break $1$-forms.
\end{itemize}

\noindent The {\sl discriminant} $\Delta(\W)$ of $\W$ is the divisor on $M$ defined locally by $\Delta(\omega_i)=0$, where $\Delta(\omega_i)$ is the discriminant of the $d$-symmetric $1$-form~$\omega_{i}\in\mathrm{Sym}^{d}\Omega^{1}_{M}(U_{i})$, \emph{cf.} \cite[Chapter~1, \S 1.3.4]{PP15}. The support of $\Delta(\W)$ consists of the points that do not satisfy condition~\textit{(2)}. When $d=1$ this condition is always satisfied and we recover the usual definition of a holomorphic foliation $\F$ on $M.$

\noindent The web $\W$ is regular if its discriminant $\Delta(\W)$ is empty.

\noindent We say that $\W$ is {\sl decomposable} if there are global webs $\W_1$ and $\W_2$ on $M$ sharing no common subwebs such that $\W$ is the superposition of $\W_1$ and $\W_2$; we then write $\W=\W_1\boxtimes\W_2.$ Otherwise $\W$ is said to be {\sl irreducible}. It is said to be {\sl completely decomposable} if there exist global foliations $\F_1,\ldots,\F_d$ on $M$ such that $\W=\F_1\boxtimes\cdots\boxtimes\F_d$.

\noindent Let $\Gamma\subset\Delta(\W)$ be an irreducible component of the discriminant of $\W.$ We say that $\Gamma$ is {\sl invariant} (resp. {\sl totally invariant}) by $\W$ if, on the regular part of $\Gamma$, we have $\mathrm{T}\hspace{0.2mm}\Gamma\subset\mathrm{T}\W|_{\Gamma}$ (resp. $\mathrm{T}\hspace{0.2mm}\Gamma=\mathrm{T}\W|_{\Gamma}$). When $\W$ is a germ of an irreducible web, these two notions coincide. For more details on this subject, \emph{see}~\cite{PP15}.

\noindent When $M=\pp$, a $d$-web is given by a polynomial $d$-symmetric $1$-form $\omega=\sum_{i+j=d}a_{ij}(x,y)\mathrm{d}x^{i}\mathrm{d}y^{j}$, with isolated zeros and non identically zero discriminant. Thus, every $d$-web on $\pp$ can be described in an affine~chart~$(x,y)$ by a polynomial differential equation $F(x,y,y')=0$ of degree $d$ in $y'.$

\subsection{Curvature and flatness}\label{subsec:courbure-platitude}

Let us briefly recall the definition of the curvature of a $d$-web with~$d\geq~\hspace{-0.5mm}3.$ First, assume that $d=3$ and consider the case of a completely decomposable $3$-web $\W$ on $(\C^2,0),$ \break $\W=\F_1\boxtimes\F_2\boxtimes\F_3.$ For~$i=1,2,3$,~let~$\omega_{i}$ be a $1$-form with an isolated singularity at $0$ defining the foliation~$\F_{i}.$ Without loss of generality, we can assume that the $1$-forms $\omega_i$ satisfy $\omega_1+\omega_2+\omega_3=0.$ It~can~be~shown that there exists a~meromorphic $1$-form $\eta(\W)$, called the {\sl fundamental form of $\W$}, such that $\mathrm{d}\omega_i=\eta(\W)\wedge\omega_i$~for~$i=1,2,3.$ This $1$-form is well-defined up to addition of a closed logarithmic $1$-form $\dfrac{\mathrm{d}g}{g}$ with $g\in\mathcal{O}^*(\mathbb{C}^{2},0).$ The curvature of $\W$ is then defined as the $2$-form $K(\W)=\mathrm{d}\,\eta(\W).$

\noindent Now, if $\W$ is a completely decomposable $d$-web, $\W=\F_1\boxtimes\cdots\boxtimes\F_d$ with $d>3$, we define the curvature $K(\W)$ of $\W$ as the sum of the curvatures of all its $3$-subwebs.

\noindent It is easy to check that $K(\W)$ is a meromorphic $2$-form with poles along the discriminant $\Delta(\W)$ of $\W$, canonically associated to $\W.$

\noindent Finally, if $\W$ is an arbitrary $d$-web on a complex surface $M$, then its pull-back by a suitable ramified \textsc{Galois} covering is completely decomposable. The curvature of this new web is invariant by the action of the \textsc{Galois} group and descends to a global meromorphic $2$-form $K(\W)$ on $M$, with poles along the discriminant of $\W$ (\emph{see}~\cite{MP13}).

\noindent A web $\W$ is said to be {\sl flat} if its curvature $K(\W)$ vanishes identically. When $M=\pp$, a web $\W$ is flat if and only if its curvature $K(\W)$ is holomorphic over the generic points of the irreducible components~of~$\Delta(\W).$

%==================================================================================================================================================================

\section{Local study of planar webs near a totally invariant curve}

\subsection{Strongly and weakly invariant curves by an irreducible $\nu$-web}\label{subsec:fort-faible-invariance}

Let $\W_\nu$ be an irreducible $\nu$-web on $(\C^2,0)$ admitting an irreducible invariant curve $C$. Since this web is irreducible, its monodromy group is cyclic of order $\nu$. Thus, in a neighborhood $U$ of a generic point of $C$, there exists a cyclic ramified \textsc{Galois} covering $\pi\colon\tilde{U}\to U$ of degree $\nu$ such that the pull-back $\pi^*\W_\nu$ decomposes completely~into~$\nu$~holomorphic foliations on $\tilde U$, and
\smallskip

\begin{itemize}
\item either the curve $\tilde{C}=\pi^{-1}(C)$ is totally invariant by $\pi^*\W_\nu$, in which case we say that $C$ is \textsl{strongly invariant by} $\W_\nu$;
\smallskip

\item or $\pi^*\W_\nu$ is transverse to $\tilde{C}$, in which case we say that $C$ is \textsl{weakly invariant by} $\W_\nu$.
\end{itemize}
\smallskip

\noindent More precisely, let us choose a local coordinate system $(z,w)$ on $U$ such that $C\cap U=\{w=0\}.$ In this system, $\W_\nu$ is defined by a $\nu$-symmetric $1$-form $\omega$ of type
\begin{align}\label{equa:nu-forme-omega}
&\omega=\mathrm{d}w^{\nu}+wa_{\nu-1}(z,w)\mathrm{d}w^{\nu-1}\mathrm{d}z+\cdots
+wa_1(z,w)\mathrm{d}w\mathrm{d}z^{\nu-1}+w^{\kappa}a_0(z,w)\mathrm{d}z^{\nu},
\end{align}
with $\kappa\in\N^*$ and $a_0(z,0)\neq0.$ The ramified covering $\pi:\tilde U\to U$ is given by $(z,w)=\pi(x,y)=(x,y^\nu),$ and $\tilde{C}=\pi^{-1}(C)=\{y=0\}.$ The above description of the $\nu$-web $\pi^*\W_\nu$ is ensured by assertion~\textit{\textbf{2.}} of the following lemma.

\begin{lem}\label{lem:multiplicite-Delta-W-nu-C-irreductible}
{\sl

\textit{\textbf{1.}} For $\nu\geq2$, we have $\mathrm{mult}\left(\Delta(\W_{\nu}),C\right)\geq\kappa(\nu-1),$ and equality holds if and only if $\gcd(\nu,\kappa)=1.$ In particular, $\mathrm{mult}\left(\Delta(\W_{\nu}),C\right)=\nu-1$ if and only if $\kappa=1.$
\smallskip

\noindent\textit{\textbf{2.}} We have the following alternative:
\smallskip

\begin{itemize}
\item If $\kappa<\nu,$ then $\pi^*\W_\nu=\boxtimes_{j=1}^{\nu}\F_j$, where $\F_j$ is a foliation transverse to $\tilde{C}$ and defined by a $1$-form $\omega_j$ of type
\[
\omega_j=\mathrm{d}x+\zeta^{-j\kappa}y^{\nu-\kappa-1}f(x,\zeta^jy)\mathrm{d}y,
\qquad\qquad f(x,0)\not\equiv0,\hspace{2mm}\zeta=\exp(\tfrac{2\mathrm{i}\pi}{\nu}).
\]

\item If $\kappa\geq\nu,$ then $\pi^*\W_\nu=\boxtimes_{j=1}^{\nu}\F_j$, where $\F_j$ is a foliation admitting $\tilde{C}$ as an invariant curve and given by a $1$-form $\theta_j$ of type
\[
\hspace{-2cm}\theta_j=\mathrm{d}y+\zeta^{j\kappa}y^{\kappa-\nu+1}h(x,\zeta^jy)\mathrm{d}x,
\qquad\qquad h(x,0)\not\equiv0.
\]
\end{itemize}
}
\end{lem}

\begin{proof}[\sl Proof]
The web $\W_\nu$ being irreducible, the polynomial defining $\W_\nu$ in $\mathbb{C}((z))\{w,\frac{\mathrm{d}w}{\mathrm{d}z}\}$ has the following \textsc{Puiseux} parametrizations:
\begin{align}\label{equa:puiseux}
\frac{\mathrm{d}w}{\mathrm{d}z}=\sum_{l=r}^{\infty}c_{l-r}(z)\left(\zeta^{j}w^{\frac{1}{\nu}}\right)^{l}=\zeta^{jr}c_{0}(z)w^{\frac{r}{\nu}}+\cdots,
\qquad c_{0}(z)\not\equiv 0,\quad\zeta=\exp(\tfrac{2\mathrm{i}\pi}{\nu}),\quad j=1,\ldots,\nu.
\end{align}
From (\ref{equa:nu-forme-omega}) and (\ref{equa:puiseux}), we obtain
\begin{align*}
\prod\limits_{j=1}^{\nu}(\zeta^{jr}c_{0}(z)w^{\frac{r}{\nu}}+\ldots)=(-1)^{r(\nu+1)}w^{r}c_{0}(z)^{\nu}+\ldots=w^{\kappa}a_{0}(z,w).
\end{align*}
Since $a_0(z,0)\neq0,$ we deduce that $r=\kappa$. Thus we have
\begin{align*}
\Delta(\W_{\nu})
&=\prod\limits_{i\neq j}\Big((\zeta^{i\kappa}c_{0}(z)w^{\frac{\kappa}{\nu}}+\cdots)-(\zeta^{j\kappa}c_{0}(z)w^{\frac{\kappa}{\nu}}+\cdots)\Big)\\
&=w^{\kappa(\nu-1)}c_{0}(z)^{\nu(\nu-1)}\prod\limits_{i\neq j}(\zeta^{i\kappa}-\zeta^{j\kappa})+\cdots.
\end{align*}
It follows that $\mathrm{mult}\left(\Delta(\W_{\nu}),C\right)\geq\kappa(\nu-1),$ with equality if and only if $\zeta^{i\kappa}\neq\zeta^{j\kappa}$ for all $i\neq j$, that is, if~and~only~if $\gcd(\nu,\kappa)=1$, since $\gcd(\nu,\kappa)=\dfrac{\nu}{\#\{\zeta^{i\kappa}\hspace{1mm}\big\vert\hspace{1mm} i=1,\ldots,\nu\}}.$ Hence the first assertion holds.
\smallskip

\noindent For the second assertion, it suffices to apply $\pi$ to~(\ref{equa:puiseux}) and use the equality $r=\kappa.$
\end{proof}

\vspace{-3mm}
\begin{rem}\label{rem:fort-faible-invariance}
\begin{itemize}
\item []\hspace{-0.8cm}(i) Saying that $C$ is strongly (resp. weakly) invariant by $\W_\nu$ is equivalent to $\kappa\geq\nu$ (resp. $\kappa<\nu$).

\item [(ii)] In the case of a foliation $\F$, {\it i.e.} when $\nu=1$, every $\F$-invariant curve is strongly invariant by $\F$.

\item [(iii)] For $\nu\geq2$, if $C$ has minimal multiplicity $\nu-1$ in $\Delta(\mathcal W_\nu)$, which is equivalent to $\kappa=1$, then $C$ is always weakly invariant by $\W_\nu.$
\end{itemize}
\end{rem}

\noindent As we have seen above, the integer $\kappa$ defined via the $\nu$-symmetric $1$-form $\omega$ coincides with the numerator of the first \textsc{Puiseux} exponent in the parametrizations~(\ref{equa:puiseux}), yielding
\begin{align}\label{equa:puiseux-kappa}
\frac{\mathrm{d}w}{\mathrm{d}z}=\zeta^{j\kappa}c_0(z)w^{\frac{\kappa}{\nu}}+\cdots,\qquad c_{0}(z)\not\equiv 0,\quad j=1,\ldots,\nu.
\end{align}

\noindent Let us consider another local coordinate system $(z',w')$ such that $C=\{w'=0\}$. We can then write
\[
\hspace{-3.1cm}z=\varphi(z',w'), \qquad w=u(z',w')w',
\]
for some holomorphic functions $\varphi$ and $u$ satisfying
\[
\hspace{-1.3cm}0\neq\left.\det\left(\frac{\partial(z,w)}{\partial(z',w')}\right)\right|_{w'=0}=\partial_{z'}\varphi(z',0)u(z',0).
\]

\noindent In this new coordinate system, the~\textsc{Puiseux}~parametrizations~(\ref{equa:puiseux-kappa}) take the form
\begin{align}\label{equa:puiseux-kappa-nouveau-systeme}
\frac{\mathrm{d}w'}{\mathrm{d}z'}=
\frac{\zeta^{j\kappa}c_0(\varphi(z',0))\partial_{z'}\varphi(z',0)}{u(z',0)^{1-\frac{\kappa}{\nu}}}w'^{\frac{\kappa}{\nu}}
-\frac{\partial_{z'}u(z',0)}{u(z',0)}w'
+\cdots,\qquad\quad j=1,\ldots,\nu.
\end{align}

\noindent When $C$ is weakly invariant by $\W_\nu$, {\it i.e.} if $\kappa<\nu$, we see from~(\ref{equa:puiseux-kappa-nouveau-systeme}) that the first nonzero exponent in the \textsc{Puiseux} series of $\frac{\mathrm{d}w'}{\mathrm{d}z'}$ at $w'=0$ remains equal to $\frac{\kappa}{\nu}$. It follows that the integer $\kappa\in\{1,\ldots,\nu-1\}$ does not depend on the choice of local coordinate system and is therefore intrinsically attached to the pair $(\W_\nu,C)$. We call this integer $\kappa$ the {\sl \textsc{Puiseux} index of $\W_\nu$ relative to $C$} and we denote it by $\mathfrak{i}(\mathcal W_\nu,C).$ The ratio $\frac{\kappa}{\nu}$ is called the {\sl \textsc{Puiseux} exponent of $\W_\nu$ relative to $C$} and is denoted by $\rho(\W_\nu,C).$

\begin{rem}
If $\W_n$ is an $n$-web on $(\C^2,0)$ (not necessarily irreducible) admitting a totally invariant irreducible curve $C$, then we have the decomposition
$$\W_{n}=\W_{n}^{\rm{str}}\boxtimes\W_{n}^{\rm{wk}},$$
where $\W_n^{\mathrm{str}}$ (resp. $\W_n^{\mathrm{wk}}$) denotes the web consisting of the irreducible subwebs of $\W_n$ having $C$ as a strongly (resp. weakly) invariant curve.
\end{rem}

\subsection{Non-degenerate webs along a totally invariant curve}

We introduce the following definition.

\begin{defin}\label{defin:W-nu-non-degenere-le-long-C}
Let $\W_\nu$ be a $\nu$-web on $(\C^2,0)$ admitting a totally invariant irreducible curve $C.$ Let $\W_\nu=\boxtimes_{\alpha=1}^{r}\W_{\nu_{\alpha}}$ be the decomposition of $\W_\nu$ into its irreducible components. We say that $\W_\nu$ is {\sl non-degenerate~along~$C$} if $C$ is weakly invariant by each $\W_{\nu_\alpha}$ and if, denoting $\kappa_\alpha=\mathfrak{i}(\W_{\nu_{\alpha}},C)$ and~$\rho_\alpha=\rho(\W_{\nu_{\alpha}},C)=\frac{\kappa_\alpha}{\nu_\alpha},$ the~following conditions are satisfied:
\begin{itemize}
\item [($\mathfrak{a}$)] for all $\alpha=1,\ldots,r$, $\gcd(\nu_\alpha,\kappa_\alpha)\leq2$;

\item [($\mathfrak{b}$)] for all $\alpha\neq \beta$, if $\rho_\alpha=\rho_\beta$, then $\gcd(\nu_\alpha,\kappa_\alpha)=\gcd(\nu_\beta,\kappa_\beta)=1$;

\item [($\mathfrak{c}$)] there are no three distinct indices $\alpha,\beta,\gamma\in\{1,\ldots,r\}$ such that $\rho_\alpha=\rho_\beta=\rho_\gamma.$
\end{itemize}
\end{defin}

\begin{rem}\label{rem:tissu-non-degenere}
Let $\W_\nu$ be an irreducible $\nu$-web on $(\C^2,0)$ having an irreducible invariant curve $C$. Saying that $\W_\nu$ is non-degenerate along $C$ amounts to saying that $C$ is weakly invariant by $\W_\nu$ and that $\gcd(\nu,\kappa)\leq2$, where~$\kappa=\mathfrak{i}(\W_{\nu},C).$
\smallskip

\noindent In particular:
\smallskip

\begin{itemize}
\item [(i)] The web $\W_\nu$ is always non-degenerate along $C$ if $C$ is weakly invariant by $\W_\nu$ with \textsc{Puiseux} index $\kappa\in\{1,2,\nu-1,\nu-2\}.$

\item [(ii)] If $C$ has minimal multiplicity $\nu-1$ in $\Delta(\W_\nu)$, then $\W_\nu$ is non-degenerate along $C$ (\emph{cf.}  Remark~\ref{rem:fort-faible-invariance}~(iii)).

\item [(iii)] When $\nu$ is prime or equal to $4$, $\W_\nu$ is non-degenerate along every curve which is weakly invariant~by~$\W_\nu.$
\end{itemize}
\end{rem}

\begin{eg}\label{eg:tissu-non-degenere}
Let $\W_{\nu_\alpha},\alpha=1,\ldots,r,$ be irreducible $\nu_\alpha$-webs on $(\C^2,0)$, with $\nu_\alpha\geq2$, sharing a common irreducible invariant curve $C$. We assume that:
\begin{itemize}
\item [\texttt{1.}] $\mathrm{mult}\left(\Delta(\W_{\nu_\alpha}),C\right)=\nu_\alpha-1$, for all $\alpha=1,\ldots,r$;

\item [\texttt{2.}] there are no three distinct indices $\alpha,\beta,\gamma\in\{1,\ldots,r\}$ such that $\nu_\alpha=\nu_\beta=\nu_\gamma.$
\end{itemize}
Then the web $\W_\nu=\W_{\nu_1}\boxtimes\cdots\boxtimes\W_{\nu_r}$ is non-degenerate along $C.$
\end{eg}

\noindent The previous definition is a generalization of the situation considered in \cite[\S~2.5]{MP13}, namely that of irreducible $\nu$-webs $\W_\nu$ admitting an irreducible invariant curve of minimal multiplicity $\nu-1$ in $\Delta(\W_\nu).$

%==================================================================================================================================================================

\section{Holomorphy of the curvature of planar webs along an invariant curve}

\noindent In this section, we study the holomorphy of the curvature of webs on $(\mathbb{C}^2,0)$ along an irreducible invariant curve. We consider a $d$-web of the form
\begin{align}\label{equa:W-Wn-W-d-n}
\W=\W_n\boxtimes\W_{d-n},
\end{align}
where $\W_n$ is an $n$-web having an irreducible totally invariant curve $C$, and $\mathcal{W}_{d-n}$ is a regular $(d-n)$-web transverse to~$C$. The main result we establish for such a web is a reduction theorem: the study of the holomorphy of $K(\W)$ along $C$ reduces to that of $K(\W_n)$; the transverse component $\W_{d-n}$ does not affect it.

\begin{thm}\label{thm:W-n-W-d-n}
{\sl Let $\W_n$ be an $n$-web on $(\C^2,0)$ admitting a totally invariant irreducible curve $C$. Let $\W_{d-n}$ be a regular $(d-n)$-web on $(\C^2,0)$ transverse to $C$. Set $\W=\W_{n}\boxtimes\W_{d-n}.$ Then $K(\W)-K(\W_n)$ is holomorphic along $C$. In particular, $K(\W)$ is holomorphic along $C$ if and only if $K(\W_n)$ is holomorphic along~$C.$}
\end{thm}

\noindent For $n=2$, we recover a special case of \cite[Theorem~1]{MP13}: the curvature $K(\W)$ is always holomorphic on $C$.

\noindent The proof of this theorem relies on the following technical lemma.

\begin{lem}\label{lem:eta-W-3-n-lambda}
{\sl Let $n_1,n_2,n_3\in\N$ and consider the $3$-web $\W$ defined by the $1$-forms
\begin{align*}
\omega_i=\mathrm{d}x+y^{n_i}h_i(x,y)\mathrm{d}y,\qquad\quad i=1,2,3.
\end{align*}
For $i\neq j$, we set
\begin{align*}
h_{ij}=
\left\{
\begin{array}{lcl}
h_{i}                    & \text{if} & n_{i}<n_{j}, \\
h_{i}-h_{j}              & \text{if}  & n_{i}=n_{j},\\
\hphantom{h_{i}}-h_{j}   & \text{if}  & n_{i}>n_{j}.
\end{array}
\right.
\end{align*}
We assume that the following conditions are satisfied:
\begin{itemize}
\item [$\bullet$] $h_{1},h_{2},h_{23}$ and $h_{31}$ do not vanish along $y=0$;

\item [$\bullet$] if $n_1=n_2=n_3$ and $h_{12}(x,0)\equiv0$ ({\it i.e.} $h_1(x,0)=h_2(x,0)$), then there exists $\lambda\in\C\setminus\{1\}$ such that $h_3(x,0)=\lambda\,h_1(x,0).$
\end{itemize}
Then the fundamental form of $\W$ has the following form
\begin{align*}
\eta(\W)=\left(\frac{g(x)}{y^{n_0+1}}+\cdots\right)\mathrm{d}x+
\left(\frac{c}{y}+\cdots\right)\mathrm{d}y,
\end{align*}
where $n_0=\min(n_1,n_2,n_3)$, $g$ is a holomorphic function of $x$, $c\in\C$ and the dots denote higher order terms~in~$y.$
}
\end{lem}

\begin{proof}[\sl Proof]
By arguing as in~\cite[Lemma~2.2]{MP13}, we obtain that $\eta(\W)=A(x,y)\mathrm{d}x+B(x,y)\mathrm{d}y,$ where
\begin{align*}
A=\frac{
\left|
\begin{matrix}
\partial_x(\delta_{12} h_3)y^{n_3}-\partial_y \delta_{12} & -\delta_{12}\\
\partial_x(\delta_{23} h_1)y^{n_1}-\partial_y \delta_{23} & -\delta_{23}
\end{matrix}
\right|}
{\delta},
&&
B=\frac{
\left|
\begin{matrix}
\delta_{12}y^{n_3}h_3 & \partial_x(\delta_{12} h_3)y^{n_3}-\partial_y\delta_{12}\\
\delta_{23}y^{n_1}h_1 & \partial_x(\delta_{23} h_1)y^{n_1}-\partial_y\delta_{23}
\end{matrix}
\right|
}{\delta},
\end{align*}
with $\delta_{ij}=y^{n_j}h_j-y^{n_i}h_i$\, and \,$\delta = \delta_{12} \delta_{23} \delta_{31}.$

\noindent When $h_{12}(x,0)\not\equiv0$, we have $y\not|\hspace{1.5mm}h_{12}h_{23}h_{31}$; if $n_1\leq n_2\leq n_3$, an argument from~\cite[Lemma~2.2]{MP13} shows that
\begin{align*}
A(x,y)=\frac{n_2-n_1}{h_{31}(x,0)}\frac{1}{y^{n_1+1}}+\cdots,
\qquad
B(x,y)=\frac{n_1}{y}+\cdots.
\end{align*}

\noindent Assume now that $h_{12}(x,0)\equiv0$. Then $n_1=n_2$ and $h_1(x,0)=h_2(x,0)$. Let us write $h_2(x,y)=h_1(x,y)+y^ku(x,y)$ with $k\geq1$ and $u(x,0)\not\equiv0.$ Consider the case $n_3\leq n_1=n_2$. A direct computation gives
\begin{Small}
\begin{align*}
&\delta(x,y)A(x,y)=(n_1+k-n_3)u(x,0)h_{31}(x,0)y^{n_1+k+n_3-1}+\cdots,&&
\delta(x,y)=-u(x,0)h_{31}(x,0)^2y^{n_1+k+2n_3}+\cdots,\\
&\delta(x,y)B(x,y)=u(x,0)h_{31}(x,0)\Big(kh_3(x,0)-(n_3+k)h_{31}(x,0)\Big)y^{n_1+k+2n_3-1}+\cdots.
\end{align*}
\end{Small}
\hspace{-1mm}It follows that
\begin{align*}
A(x,y)=\frac{n_3-n_1-k}{h_{31}(x,0)}\frac{1}{y^{n_3+1}}+\cdots,
\qquad
B(x,y)=\left(n_3+k-\frac{kh_3(x,0)}{h_{31}(x,0)}\right)\frac{1}{y}+\cdots.
\end{align*}

\noindent If $n_3<n_1=n_2$, then $A(x,y)=\frac{n_3-n_1-k}{h_{3}(x,0)}\frac{1}{y^{n_3+1}}+\cdots$\, and \,$B(x,y)=\frac{n_3}{y}+\cdots.$

\noindent If $n_1=n_2=n_3$, we have by hypothesis $h_3(x,0)=\lambda h_1(x,0)$ with $\lambda\neq1$, hence
\begin{align*}
A(x,y)=\frac{k}{(1-\lambda)h_1(x,0)}\frac{1}{y^{n_1+1}}+\cdots,
\qquad
B(x,y)=\left(n_1+\frac{k}{1-\lambda}\right)\frac{1}{y}+\cdots.
\end{align*}

\noindent The case $n_3>n_1=n_2$ reduces to the case $n_1=n_2=n_3$ with $\lambda=0$, noting that $\omega_3=\mathrm{d}x+y^{n_1}\tilde{h}_3(x,y)\mathrm{d}y$, where $\tilde{h}_3(x,y)=y^{n_3-n_1}h_3(x,y).$
\end{proof}

\begin{proof}[\sl Proof of Theorem~\ref{thm:W-n-W-d-n}]
We will argue as in the proof of~\cite[Proposition~2.6]{MP13}. Consider again the local coordinate system $(U,(z,w))$ where $C\cap U=\{w=0\}.$ We can write
\begin{align*}
\mathrm{T}\W_{d-n}|_{U}=\left\{\prod\limits_{l=1}^{d-n}(\mathrm{d}z+g_{l}(z,w)\mathrm{d}w)=0\right\}.
\end{align*}

\textbf{\textit{1.}} We deal with the case $\W_n=\W_{n}^{\rm{wk}}$. First assume that $\W_n$ is irreducible; then $C$ is weakly invariant~by~$\W_n$. Passing to the ramified covering $\pi(x,y)=(z,w)=(x,y^{n})$, Lemma~\ref{lem:multiplicite-Delta-W-nu-C-irreductible} ensures that $\pi^{*}\W_{n}=\boxtimes_{j=1}^{n}\F_j$ and $\pi^{*}\W_{d-n}=\boxtimes_{l=1}^{d-n}\G_{l},$ where
\begin{align*}
\F_j\hspace{0.1mm}:\hspace{0.1mm}\mathrm{d}x+\zeta^{-j\kappa}y^{n-\kappa-1}f(x,\zeta^jy)\mathrm{d}y=0,&&
\G_{l}\hspace{0.1mm}:\hspace{0.1mm}\mathrm{d}x+ny^{n-1}g_{l}(x,y^{n})\mathrm{d}y=0,
\end{align*}
with $1\leq\kappa<n$, $\zeta=\exp(\tfrac{2\mathrm{i}\pi}{n})$ and $f(x,0)\not\equiv0.$
\smallskip

\noindent We then have
\begin{align*}
K(\pi^*\W)-K(\pi^*\W_n)-K(\pi^*\W_{d-n})
&=\sum_{\substack{1\le j<j'\le n\\ 1\le l\le d-n}}K(\F_j\boxtimes\F_{j'}\boxtimes\G_l)+\sum_{\substack{1\le j\le n\\ 1\le l<l'\le d-n}}K(\F_j\boxtimes\G_l\boxtimes\G_{l'})\\
&=\sum_{\substack{1\le j<j'\le n\\ 1\le l\le d-n}}\mathrm{d}\eta_{jj'l}+\sum_{\substack{1\le j\le n\\ 1\le l<l'\le d-n}}\mathrm{d}\eta_{jll'},
\end{align*}
where $\eta_{jj'l}:=\eta(\F_j\boxtimes\F_{j'}\boxtimes\G_l)$\,\, and\,\, $\eta_{jll'}:=\eta(\F_j\boxtimes\G_l\boxtimes\G_{l'}).$

\noindent Denoting by $\varphi_{\ell}(x,y)=(x,\zeta^{\ell}y)$, $\ell=1,\ldots,n,$ the \textsc{Deck} transformations of $\pi$, we deduce that
\begin{align*}
K(\pi^*\W)-K(\pi^*\W_n)-K(\pi^*\W_{d-n})
&=\pi^*\Big(K(\W)-K(\W_n)-K(\W_{d-n})\Big)\\
&=\frac{1}{n}\sum_{\ell=1}^{n}\varphi_\ell^*\pi^*\Big(K(\W)-K(\W_n)-K(\W_{d-n})\Big)\\
&=\sum_{\substack{1\le j<j'\le n\\1\le l\le d-n}}\mathrm{d}\left(\frac{1}{n}\sum_{\ell=1}^{n}\varphi_\ell^*\eta_{jj'l}\right)
+\sum_{\substack{1\le j\le n\\ 1\le l<l'\le d-n}}\mathrm{d}\left(\frac{1}{n}\sum_{\ell=1}^{n}\varphi_\ell^*\eta_{jll'}\right).
\end{align*}
\noindent Note that if $m\not\equiv0\mod n$ then $\frac{1}{n}\sum_{\ell=1}^{n}\varphi_{\ell}^*(y^{m}\mathrm{d}x)=0$, and if $m\equiv -1 \mod n$ then $\frac{1}{n}\sum_{\ell=1}^{n}\varphi_{\ell}^*(y^{m}\mathrm{d}y)=y^{m}\mathrm{d}y$. Moreover, if we write
\begin{align*}
\eta_{jj'l}=A_{jj'l}(x,y)\mathrm{d}x+B_{jj'l}(x,y)\mathrm{d}y
&&\text{and}&&
\eta_{jll'}=A_{jll'}(x,y)\mathrm{d}x+B_{jll'}(x,y)\mathrm{d}y,
\end{align*}
Lemma~\ref{lem:eta-W-3-n-lambda} shows that each of $A_{jj'l}$ and $A_{jll'}$ has poles of order $\leq n-\kappa\leq n-1$ along~$y=0$, and that each of $B_{jj'l}$ and $B_{jll'}$ is logarithmic along $y=0$, with constant residue. It follows that~$\mathrm{d}\left(\frac{1}{n}\sum_{\ell=1}^{n}\varphi_{\ell}^*\eta_{jj'l}\right)$ and $\mathrm{d}\left(\frac{1}{n}\sum_{\ell=1}^{n}\varphi_{\ell}^*\eta_{jll'}\right)$ are holomorphic along $\tilde{C}=\{y=0\}$, and hence so is $K(\pi^*\W)-K(\pi^*\W_n)-K(\pi^*\W_{d-n})$. Since~$\W_{d-n}$ is regular, $K(\W_{d-n})$ is holomorphic along $C$, and we deduce that $K(\W)-K(\W_{n})$ is also holomorphic along $C.$
\smallskip

\noindent Now assume that $\W_n=\boxtimes_{\alpha=1}^{r}\W_{\nu_{\alpha}}$ with $r\geq2$, each $\W_{\nu_{\alpha}}$ being an irreducible $\nu_{\alpha}$-web having $C$ as a weakly invariant curve. Consider the ramified covering $\pi\hspace{1mm}\colon(x,y)\mapsto(x,y^{\tilde{n}})$, where $\tilde{n}:=\prod_{\alpha=1}^{r}\nu_{\alpha}.$ The pull-back of $\W_{d-n}$ by $\pi$ writes as $\pi^{*}\W_{d-n}=\boxtimes_{l=1}^{d-n}\G_{l},$ where
\[
\G_{l}\hspace{0.1mm}:\hspace{0.1mm}\mathrm{d}x+\tilde{n}y^{\tilde{n}-1}g_{l}(x,y^{\tilde{n}})\mathrm{d}y=0.
\]
As for $\pi^{*}\W_{n}$, we have $\pi=\pi_{\alpha}\circ\pi^{'}_{\alpha}$, where $\pi_{\alpha}\hspace{1mm}\colon(x,y)\mapsto(x,y^{\nu_{\alpha}})$ and $\pi^{'}_{\alpha}\hspace{1mm}\colon(x,y)\mapsto(x,y^{\hspace{0.3mm}\mu_\alpha})$ with $\mu_\alpha:= \frac{\tilde{n}}{\nu_\alpha}$. By~Lemma~\ref{lem:multiplicite-Delta-W-nu-C-irreductible}, $\pi_\alpha^*\W_{\nu_{\alpha}}=\boxtimes_{j=1}^{\nu_{\alpha}}\F_{j,0}^{\alpha}$, where
\[
\F_{j,0}^{\alpha}\,:\,\mathrm{d}x+\zeta_{\alpha}^{-j\kappa_{\alpha}}y^{\nu_{\alpha}-\kappa_{\alpha}-1}f_\alpha(x,\zeta_{\alpha}^{j}y)\mathrm{d}y=0,
\]
with $1\leq\kappa_{\alpha}<\nu_\alpha$, $\zeta_{\alpha}=\exp(\tfrac{2\mathrm{i}\pi}{\nu_{\alpha}})$ and $f_\alpha(x,0)\not\equiv0$. It follows that $\pi^*\W_{n}=\boxtimes_{\alpha=1}^{r}\boxtimes_{j=1}^{\nu_{\alpha}}\F_{j}^{\alpha},$ where
\[
\F_{j}^{\alpha}\,:\,\mathrm{d}x
+\mu_{\alpha}\zeta_{\alpha}^{-j\kappa_{\alpha}}y^{\tilde{n}-\tilde{\kappa}_{\alpha}-1}f_{\alpha}(x,\zeta_{\alpha}^{j}y^{\hspace{0.3mm}\mu_{\alpha}})\mathrm{d}y=0,
\]
with $\tilde{\kappa}_{\alpha}:=\mu_{\alpha}\kappa_{\alpha}$,\, $1\leq\tilde{\kappa}_{\alpha}<\tilde{n}$.

\noindent We are thus reduced to a situation analogous to the case where $\W_n$ was irreducible; the same argument shows that $K(\W)-K(\W_{n})$ is holomorphic along $C$.
\smallskip

\textbf{\textit{2.}} We treat the general case. Write $\W_{n}^{\rm{str}}=\boxtimes_{\alpha=1}^{r}\W_{\nu_{\alpha}}$ and $\W_{n}^{\rm{wk}}=\boxtimes_{\beta=1}^{s}\W_{\nu^{'}_{\beta}}$, where each $\W_{\nu_{\alpha}}$ (resp. $\W_{\nu^{'}_{\beta}}$) is an irreducible $\nu_{\alpha}$-web (resp. $\nu^{'}_{\beta}$-web) admitting $C$ as a strongly (resp. weakly) invariant curve. Passing to the covering $\pi\hspace{1mm}\colon(x,y)\mapsto(z,w)=(x,y^{\tilde{n}})$, where $\tilde{n}=\displaystyle\prod_{\alpha=1}^{r}\nu_{\alpha}\displaystyle\prod_{\beta=1}^{s}\nu^{'}_{\beta}$, we obtain that $\pi^*\W_{n}^{\rm{str}}=\boxtimes_{i=1}^{n_1}\mathscr{F}_i$, $\pi^*\W_{n}^{\rm{wk}}=\boxtimes_{j=1}^{n_2}\F_j$ and $\pi^*\W_{d-n}=\boxtimes_{l=1}^{d-n}\G_{l}$, where the $\mathscr{F}_i$ are foliations having $\tilde{C}=\{y=0\}$ as an invariant curve, and the $\F_{j}$ and $\G_{l}$ are foliations transverse to $\tilde{C},$ $n_1=\sum_{\alpha=1}^{r}\nu_\alpha$, and $n_2=\sum_{\beta=1}^{s}\nu^{'}_{\beta}.$
\smallskip

\noindent Then we have
\begin{small}
\begin{align*}
K(\pi^*\W)-K(\pi^*\W_n)&=K\big(\pi^*(\W_{n}^{\rm{wk}}\boxtimes\W_{d-n})\big)-K(\pi^*\W_{n}^{\rm{wk}})
+\sum_{\substack{1\le i<i'\le n_1\\ 1\le l\le d-n}}K(\mathscr{F}_i \boxtimes\mathscr{F}_{i'}\boxtimes \G_l)
+\sum_{\substack{1\le i\le n_1\\ 1\le j\le n_2\\ 1\le l\le d-n}}K(\mathscr{F}_i\boxtimes\F_j\boxtimes\G_l)\\
&\hspace{4mm}
+\sum_{\substack{1\le i\le n_1\\ 1\le l<l'\le d-n}}K(\mathscr{F}_i \boxtimes\G_l\boxtimes\G_{l'}).
\end{align*}
\end{small}
\hspace{-1mm}By case \textbf{\textit{1.}}, $K\big(\pi^*(\W_{n}^{\rm{wk}}\boxtimes\W_{d-n})\big)-K(\pi^*\W_{n}^{\rm{wk}})=\pi^*\Big(K(\W_{n}^{\rm{wk}}\boxtimes\W_{d-n})-K(\W_{n}^{\rm{wk}})\Big)$ is holomorphic~on~$\tilde{C}$. According to~\cite[Theorem~1]{MP13}, the curvatures $K(\mathscr{F}_i \boxtimes\mathscr{F}_{i'}\boxtimes \G_l)$, $K(\mathscr{F}_i\boxtimes\F_j\boxtimes\G_l)$ and $K(\mathscr{F}_i \boxtimes\G_l\boxtimes\G_{l'})$ are~also~holomorphic on $\tilde{C}$. It follows that $K(\pi^*\W)-K(\pi^*\W_n)$ is holomorphic along $\tilde{C}$. Consequently, $K(\W)-K(\W_n)$ is holomorphic along $C.$
\end{proof}

\noindent In~\cite[Proposition~2.6]{MP13}, the authors consider a $d$-web $\W$ of the form~(\ref{equa:W-Wn-W-d-n}) and prove that, if $C$ has minimal multiplicity $n-1$ in $\Delta(\W_n)$, then the curvature of $\W$ is holomorphic along $C$. Theorem~\ref{thm:W-n-W-d-n} allows~us~to~recover~this~result: by arguing only on the component $\W_n$, while following the arguments of~\cite[Proposition~2.6]{MP13}, we obtain the holomorphy on $C$ of $K(\W_n)$, and hence that of $K(\W).$ The condition $\mathrm{mult}\left(\Delta(\W_{n}),C\right)=n-1$ implies that the web $\W_n$ is irreducible (\cite[Lemma~2.5]{MP13}) and non-degenerate along~$C$ (Remark~\ref{rem:tissu-non-degenere}~(ii)). The following proposition extends the holomorphy result of $K(\W_n)$ to non-degenerate webs along $C$, without imposing any condition on the multiplicity of $C$ in $\Delta(\W_n)$ or even assuming the irreducibility of $\W_n$.

\begin{pro}\label{pro:W-nu-non-degenere-le-long-C}
{\sl Let $\W_n$ be an $n$-web on $(\C^2,0)$ having a totally invariant irreducible curve $C$. Assume that $\W_n$ is non-degenerate along $C$. Then the curvature of $\W_n$ is holomorphic along $C.$
}
\end{pro}

\begin{proof}[\sl Proof]
We argue as in case~\textbf{\textit{1.}} of the proof of Theorem~\ref{thm:W-n-W-d-n}, using the same notation. Writing $\W_n=\boxtimes_{\alpha=1}^{r}\W_{\nu_\alpha}$ and passing to the ramified covering $\pi\hspace{1mm}\colon(x,y)\mapsto(z,w)=(x,y^{\tilde{n}})$, where $\tilde{n}=\prod_{\alpha=1}^{r}\nu_\alpha$, we obtain that $\pi^*\W_{n}=\boxtimes_{\alpha=1}^{r}\boxtimes_{j=1}^{\nu_{\alpha}}\F_{j}^{\alpha},$ where each  $\F_j^\alpha$ is a foliation transverse to $\tilde{C}=\{y=0\}$ and defined by
\[
\omega_j^\alpha=\mathrm{d}x+\mu_{\alpha}\zeta_{\alpha}^{-j\kappa_{\alpha}}y^{n_\alpha}f_{\alpha}(x,\zeta_{\alpha}^{j}y^{\hspace{0.3mm}\mu_{\alpha}})\mathrm{d}y,
\]
with $n_\alpha=\tilde{n}-\tilde{\kappa}_\alpha-1$, $\tilde{\kappa}_\alpha=\mu_\alpha\kappa_\alpha=\tilde{n}\rho_\alpha$, $\rho_\alpha=\frac{\kappa_\alpha}{\nu_\alpha}<1$, $1\le\tilde{\kappa}_\alpha<\tilde{n}$\, and \,$f_\alpha(x,0)\not\equiv0.$
\smallskip

\noindent Let us denote $\eta^{\alpha\beta\gamma}_{jj'j''}=\eta(\F_j^\alpha\boxtimes\F_{j'}^\beta\boxtimes\F_{j''}^\gamma)$; we can write
\[
K(\pi^*\W_n)=\sum_{1\le\alpha\le\beta\le\gamma\le r} \sum_{\substack{1\le j\le\nu_\alpha\\ 1\le j'\le\nu_\beta\\ 1\le j''\le\nu_\gamma}} \mathrm{d} \eta^{\alpha\beta\gamma}_{jj'j''}.
\]

\noindent Setting $\varphi_{\ell}(x,y)=(x,\zeta^{\ell}y)$, $\ell=1,\ldots,\tilde{n},$ where $\zeta=\exp(\tfrac{2\mathrm{i}\pi}{\tilde{n}}),$ we have
\begin{align*}
K(\pi^*\W_n)
=\pi^*K(\W_n)
=\frac{1}{\tilde{n}}\sum_{\ell=1}^{\tilde{n}}\varphi_\ell^*\pi^*K(\W_n)
=\sum_{1\le\alpha\le\beta\le\gamma\le r}\sum_{\substack{1\le j\le\nu_\alpha\\1\le j'\le\nu_\beta\\1\le j''\le \nu_\gamma}}\mathrm{d}
\left(
\frac{1}{\tilde{n}}\sum_{\ell=1}^{\tilde{n}}\varphi_\ell^*\eta^{\alpha\beta\gamma}_{jj'j''}
\right).
\end{align*}

\noindent Using conditions ($\mathfrak{a}$), resp. ($\mathfrak{b}$), resp. ($\mathfrak{c}$), we easily check that the $3$-webs $\F_j^\alpha\boxtimes\F_{j'}^\alpha\boxtimes\F_{j''}^\alpha$,~resp.~$\F_j^\alpha\boxtimes\F_{j'}^\alpha\boxtimes\F_{j''}^\gamma$ with $\alpha\neq\gamma$, resp. $\F_j^\alpha\boxtimes\F_{j'}^\beta\boxtimes\F_{j''}^\gamma$ with $\alpha<\beta<\gamma$, have the form of Lemma~\ref{lem:eta-W-3-n-lambda}, so that
\[
\eta_{jj'j''}^{\alpha\beta\gamma}=\left(\frac{g_{jj'j''}^{\alpha\beta\gamma}(x)}{y^{n_{\alpha\beta\gamma}+1}}+\cdots\right)\mathrm{d}x
+\left(\frac{c_{jj'j''}^{\alpha\beta\gamma}}{y}+\cdots\right)\mathrm{d}y,
\]
where $n_{\alpha\beta\gamma}=\min(n_\alpha,n_\beta,n_\gamma)$, $g_{jj'j''}^{\alpha\beta\gamma}$ is a holomorphic function of $x$ and $c_{jj'j''}^{\alpha\beta\gamma}\in\mathbb{C}.$ Since $n_{\alpha\beta\gamma}+1\le\tilde{n}-1$,  we~deduce (\emph{cf.} proof
of Theorem~\ref{thm:W-n-W-d-n}) that $\mathrm{d}\Big(\frac{1}{\tilde{n}}\sum_{\ell=1}^{\tilde{n}}\varphi_\ell^*\eta_{jj'j''}^{\alpha\beta\gamma}\Big)$ is holomorphic along $\tilde{C}=\{y=0\}$, and~hence so is $K(\pi^*\W_n)$. As a result, $K(\W_n)$ is holomorphic along $C=\{w=0\}$.
\end{proof}

\noindent From Theorem~\ref{thm:W-n-W-d-n} and Proposition~\ref{pro:W-nu-non-degenere-le-long-C}, we obtain a generalization of Proposition~2.6 of \cite{MP13} by replacing the~minimal~multiplicity assumption $\mathrm{mult}\left(\Delta(\W_n),C\right)=n-1$ with the condition $\mathrm{mult}\left(\Delta(\W_n),C\right)<3(n-1).$

\begin{cor}\label{cor:W-nu-W-tr-reg}
{\sl Let $\W_n$ be an irreducible $n$-web on $(\C^2,0)$, with $n\geq 2$. Assume that $\Delta(\W_n)$ admits an irreducible component $C$ invariant by $\W_n$ and such that $\mathrm{mult}\left(\Delta(\W_n),C\right)<3(n-1).$ Let $\W_{d-n}$ be a regular $(d-n)$-web on $(\C^2,0)$ transverse to $C$. Then the curvature of the $d$-web $\W=\W_n\boxtimes\W_{d-n}$ is holomorphic along $C$.
}
\end{cor}

\begin{proof}[\sl Proof]
If $n=2$, \cite[Theorem~1]{MP13} ensures that the curvature of $\W$ is holomorphic along $C$.
\smallskip

\noindent Assume that $n\geq3.$ According to the first assertion of Lemma~\ref{lem:multiplicite-Delta-W-nu-C-irreductible}, we have
\[
\kappa(n-1)\leq\mathrm{mult}\left(\Delta(\W_{n}),C\right)<3(n-1),
\]
hence $\kappa<3$, and therefore $\kappa\in\{1,2\}\subset\{1,\ldots,n-1\}$. It follows that $C$ is weakly invariant by $\W_n$ and that $\gcd(n,\kappa)\leq2.$ The web $\W_n$ is then non-degenerate along $C.$ By Proposition~\ref{pro:W-nu-non-degenere-le-long-C}, $K(\W_n)$ is holomorphic on~$C$, and by Theorem~\ref{thm:W-n-W-d-n}, the same holds for $K(\W)$.
\end{proof}

\noindent When $n$ is prime or $n=4$, we can weaken the condition on the multiplicity of $C$ in $\Delta(\W_n)$ by replacing the~bound~$3(n-1)$ with the bound $n(n-1).$

\begin{cor}\label{cor:W-nu-W-tr-reg-nu-premier-4}
{\sl Let $n$ be a prime number or $n=4$. Let $\W_n$ be an irreducible $n$-web on $(\C^{2},0)$. Assume that $\Delta(\W_n)$ has an irreducible component $C$ invariant by $\W_n$ and such that $\mathrm{mult}\left(\Delta(\W_n),C\right)<n(n-1).$ Let $\W_{d-n}$ be a regular $(d-n)$-web on $(\C^2,0)$ transverse to $C$. Then the curvature of the $d$-web $\W=\W_n\boxtimes\W_{d-n}$ is holomorphic along $C$.
}
\end{cor}

\begin{proof}[\sl Proof]
By Lemma~\ref{lem:multiplicite-Delta-W-nu-C-irreductible}, we have
\[
\kappa(n-1)\leq\mathrm{mult}\left(\Delta(\W_n),C\right)<n(n-1),
\]
hence $\kappa<n$. The curve $C$ is then weakly invariant by $\W_n.$ Furthermore, the assumption on $n$ and the inequality $\kappa<n$ imply that $\gcd(n,\kappa)\leq2.$ It follows that $\W_n$ is non-degenerate along $C$, so that $K(\W_n)$ is holomorphic along $C$ by Proposition~\ref{pro:W-nu-non-degenere-le-long-C}, and so is $K(\W)$ by Theorem~\ref{thm:W-n-W-d-n}.
\end{proof}

\noindent In the setting of Theorem~\ref{thm:W-n-W-d-n}, we can ask if it is possible to reduce the study of the holomorphy of $K(\W)$ along $C$ to that of $K(\W_0)$ for some subweb $\W_0\subsetneq\W_n$. The following theorem gives a partial answer to this question.

\begin{thm}\label{thm:W-n-fort}
{\sl Let $\W_n$ be an $n$-web on $(\C^2,0)$ having a totally invariant irreducible curve $C$. Assume that $\W_{n}^{\rm{wk}}$ is non-degenerate along $C$. Let $\W_{d-n}$ be a regular $(d-n)$-web on $(\C^2,0)$ transverse to $C$, and set $\W=\W_n\boxtimes\W_{d-n}.$ Then $K(\W)-K(\W_{n}^{\rm{str}})$ is holomorphic along $C$. In particular, $K(\W)$ is holomorphic along $C$ if and only if $K(\W_{n}^{\rm{str}})$ is holomorphic along $C$.
}
\end{thm}

\noindent To prove this theorem, we need the following lemma.

\begin{lem}\label{lem:K-W-n-W-d-n}
{\sl Let $\W_n=\F_1\boxtimes\cdots\boxtimes\F_n$ be a completely decomposable $n$-web having a totally invariant irreducible curve $C.$ Let $\W_{d-n}=\F^{'}_{1}\boxtimes\cdots\boxtimes\F^{'}_{d-n}$ be a completely decomposable $(d-n)$-web transverse~to~$C.$ Set $\W=\W_n\boxtimes\W_{d-n}.$ Assume that the curvature of $\W_{d-n}$ is holomorphic along~$C.$ Then $K(\W)-K(\W_n)$ is holomorphic along $C.$ In particular, $K(\W)$ is holomorphic along $C$ if and only if $K(\W_n)$ is holomorphic along $C.$
}
\end{lem}

\begin{proof}[\sl Proof]
We have
\begin{align*}
K(\W)-K(\W_n)=K(\W_{d-n})
+\sum_{\substack{1\le i<j\le n\\ 1\le k\le d-n}}K(\F_i\boxtimes\F_j\boxtimes\F^{'}_{k})
+\sum_{\substack{1\le i\le n\\ 1\le k<k'\le d-n}}K(\F_i\boxtimes\F^{'}_{k}\boxtimes\F^{'}_{k'}).
\end{align*}
Now, $K(\F_i\boxtimes\F_j\boxtimes\F^{'}_{k})$ and $K(\F_i\boxtimes\F^{'}_{k}\boxtimes\F^{'}_{k'})$ are holomorphic along $C$ by applying \cite[Theorem~1]{MP13}. The~lemma~then follows from the assumption that $K(\W_{d-n})$ is holomorphic along $C.$
\end{proof}

\begin{proof}[\sl Proof of Theorem~\ref{thm:W-n-fort}]
Let us decompose $\W_{n}^{\rm{str}}$ and $\W_{n}^{\rm{wk}}$ into $\W_{n}^{\rm{str}}=\boxtimes_{\alpha=1}^{r}\W_{\nu_\alpha}$ and $\W_{n}^{\rm{wk}}=\boxtimes_{\beta=1}^{s}\W_{\nu^{'}_{\beta}}$, where each $\W_{\nu_{\alpha}}$ (resp. $\W_{\nu^{'}_{\beta}}$) is an irreducible $\nu_{\alpha}$-web (resp. $\nu^{'}_{\beta}$-web) having $C$ as a strongly (resp. weakly) invariant curve. We take local coordinates $(z,w)$ such that $C=\{w=0\}$. Passing to the ramified covering~$\pi\hspace{1mm}\colon(x,y)\mapsto(z,w)=(x,y^{\nu})$, where $\nu=\displaystyle\prod_{\alpha=1}^{r}\nu_{\alpha}\displaystyle\prod_{\beta=1}^{s}\nu^{'}_{\beta}$, we obtain that $\pi^*\W_{n}^{\rm{str}}=\boxtimes_{i=1}^{n_1}\F_i$ and $\pi^*(\W_{n}^{\rm{wk}}\boxtimes\W_{d-n})=\boxtimes_{j=1}^{d-n_1}\F'_j$, where the $\F_i$ are foliations having $\tilde{C}=\{y=0\}$ as an invariant curve, the $\F'_{j}$ are foliations transverse to $\tilde{C}$, and $n_1=\sum_{\alpha=1}^{r}\nu_\alpha.$

\noindent The web $\W_{n}^{\rm{wk}}$ being non-degenerate along $C$, Proposition~\ref{pro:W-nu-non-degenere-le-long-C} ensures that $K(\W_{n}^{\rm{wk}})$ is holomorphic on $C$. Since $\W_{d-n}$ is regular and transverse to $C$, Theorem~\ref{thm:W-n-W-d-n} implies that $K(\W_{n}^{\rm{wk}}\boxtimes\W_{d-n})$ is also holomorphic~on~$C.$ It follows that $K\big(\pi^*(\W_{n}^{\rm{wk}}\boxtimes\W_{d-n})\big)$ is holomorphic along $\tilde{C}.$ We can then apply Lemma~\ref{lem:K-W-n-W-d-n} to~the~web~$\pi^*\W=\pi^*\W_{n}^{\rm{str}}\boxtimes\pi^*(\W_{n}^{\rm{wk}}\boxtimes\W_{d-n})$ and deduce that $K(\pi^*\W)-K(\pi^*\W_{n}^{\rm{str}})$ is holomorphic on $\tilde{C}.$ Consequently, $K(\W)-K(\W_{n}^{\rm{str}})$ is holomorphic along $C$.
\end{proof}

\begin{rem}\label{rem:W-n-W-n-nu-W-nu}
Theorem~\ref{thm:W-n-fort} applies in particular when $\W_n$ is of the form $\W_n=\W_{n-\nu}\boxtimes\W_{\nu}$, where $\W_\nu=\W_{\nu_1}\boxtimes\cdots\boxtimes\W_{\nu_r}$ is the $\nu$-web from Example~\ref{eg:tissu-non-degenere}, and $\W_{n-\nu}=\F_1\boxtimes\cdots\boxtimes\F_{n-\nu}$ is a completely decomposable $(n-\nu)$-web leaving $C$ totally invariant. In this case, we have $\W_{n}^{\rm{str}}=\W_{n-\nu}$,\, $\W_{n}^{\rm{wk}}=\W_{\nu}$, and the curvature of the $d$-web $\W=\W_{n-\nu}\boxtimes\W_\nu\boxtimes\W_{d-n-\nu}$ is holomorphic on $C$ if and only if the curvature of the $(n-\nu)$-web $\W_{n-\nu}$ is holomorphic on $C$. For $r=1$, we recover Proposition~6.2~of~\cite{Bed24arxiv2}.
\end{rem}

%==================================================================================================================================================================

\section{Examples}\label{sec:Exemples}

\noindent We give examples showing that, in Proposition~\ref{pro:W-nu-non-degenere-le-long-C}, each of the conditions ($\mathfrak{a}$), ($\mathfrak{b}$) and ($\mathfrak{c}$) involved in the definition of non-degeneracy of $\W_n$ along $C$ is essential for the validity of this proposition.

\begin{eg}\label{eg:condition-a}
Consider the $6$-web $\W$ on $(\C^2,0)$ defined in local coordinates $(z,w)$ by
\begin{align*}
\omega=\mathrm{d}w^6-3w\big(z^2+2w\big)\mathrm{d}z^2\mathrm{d}w^4+3w^2\big(z^4-2zw-2zw^2+3w^2\big)\mathrm{d}z^4\mathrm{d}w^2
-w^3\big(z^3-3zw+w+w^2\big)^2\mathrm{d}z^6.
\end{align*}
In a neighborhood of a generic point of $C=\{w=0\}$, the slopes $p_j$ $(j=1,\ldots,6)$ of $\mathrm{T}_{(z,w)}\W$ are given by
\begin{align*}
p_j=\zeta^{3j}zw^{\frac{1}{2}}+\zeta^{5j}w^{\frac{5}{6}}+\zeta^{j}w^{\frac{7}{6}},
\qquad \text{where}\hspace{1mm}\zeta=\exp(\tfrac{2\mathrm{i}\pi}{6}).
\end{align*}
In particular, $\W$ is irreducible and admits $C$ as a weakly invariant curve with \textsc{Puiseux} index $\kappa=3.$ Hence~$\gcd(\nu,\kappa)=3>2$, so condition ($\mathfrak{a}$) is not satisfied. Therefore $\W$ is degenerate along~$C.$
\smallskip

\noindent After passing to the degree $6$ cover $\pi\hspace{1mm}\colon(x,y)\mapsto(z,w)=(x,y^6),$ an explicit computation shows that
\[
K(\pi^*\W)=\left(-\frac{4}{y}+\cdots\right)\mathrm{d}x\wedge\mathrm{d}y.
\]
Thus $K(\pi^*\W)$ is not holomorphic along $\tilde{C}=\{y=0\}$, and consequently $K(\W)$ is not holomorphic along $C=\{w=0\}.$
\end{eg}

\begin{eg}\label{eg:condition-b}
Let $\W=\W_1\boxtimes\W_2$ be the $6$-web on $(\C^2,0),$ where $\W_1$, resp. $\W_2$, is the $2$-web, resp. the~$4$-web, defined by
\begin{align*}
\omega_1=\mathrm{d}w^2-(z+1)w\mathrm{d}z^2,&&
\text{resp.}\hspace{1.5mm}
\omega_2=\mathrm{d}w^4-2w\mathrm{d}w^2\mathrm{d}z^2-4w^2\mathrm{d}w\mathrm{d}z^3-(w-1)w^2\mathrm{d}z^4.
\end{align*}
Then the curve $C=\{w=0\}$ is totally invariant by each of the webs $\W_1$ and $\W_2$. According to~\cite[Lemma~2.5]{MP13}, $C$ has minimal multiplicity $\nu_1-1=1$ in $\Delta(\W_1)$; moreover, $\W_1$ is irreducible.

\noindent As for $\W_2$, in a neighborhood of a generic point of $C$, the slopes $p_j$ $(j=1,\ldots,4)$ of $\mathrm{T}_{(z,w)}\W_2$ can be~written~as
\begin{align*}
p_j=\zeta^{2j}w^{\frac{1}{2}}+\zeta^{3j}w^{\frac{3}{4}},
\qquad \text{where}\hspace{1mm}\zeta=\exp(\tfrac{2\mathrm{i}\pi}{4})=\mathrm{i}\,;
\end{align*}
it follows that $\W_2$ is also irreducible.
\smallskip

\noindent We observe that each of the webs $\W_1$ and $\W_2$ admits $C$ as a weakly invariant curve, with respective \textsc{Puiseux} indices $\kappa_1=1$ and $\kappa_2=2.$ Then condition $(\mathfrak{b})$ fails, since $\rho_1:=\frac{\kappa_1}{\nu_1}=\frac{1}{2}$ and $\rho_2:=\frac{\kappa_2}{\nu_2}=\frac{2}{4}=\rho_1$, but~$\gcd(\nu_2,\kappa_2)=2\neq1.$ As a result, $\W$ is degenerate along $C.$
\smallskip

\noindent Let us consider the degree $4$ cover $\pi\hspace{1mm}\colon(x,y)\mapsto(z,w)=(x,y^4)$. An explicit computation leads to
\[
K(\pi^*\W)=\left(-\frac{4}{x^2y}+\cdots\right)\mathrm{d}x\wedge\mathrm{d}y,
\]
so that $K(\pi^*\W)$ is not holomorphic on $\tilde{C}=\{y=0\}$, and hence $K(\W)$ is not holomorphic on $C=\{w=0\}.$
\end{eg}

\begin{eg}\label{eg:condition-c}
Consider the $6$-web $\W=\W_1\boxtimes\W_2\boxtimes\W_3$ on $(\C^2,0),$ where $\W_1$, $\W_2$ and $\W_3$ are the $2$-webs defined in local coordinates $(z,w)$ respectively by
\begin{align*}
\omega_1=\mathrm{d}w^2-w\mathrm{d}z^2,&&
\omega_2=\mathrm{d}w^2-(w+1)w\mathrm{d}z^2,&&
\omega_3=\mathrm{d}w^2-z^2w\mathrm{d}z^2.
\end{align*}
Then, for $\alpha=1,2,3$, the curve $C=\{w=0\}$ is totally invariant by $\W_\alpha$. Lemma~2.5~of~\cite{MP13} ensures that $C$ has minimal multiplicity $1$ in $\Delta(\W_\alpha)$; furthermore, each $\W_\alpha$ is irreducible.
\smallskip

\noindent We remark that $C$ is weakly invariant by each $\W_\alpha$, with \textsc{Puiseux} index $\mathfrak{i}(\W_\alpha,C)=1.$ Setting $\rho_\alpha=\rho(\W_\alpha,C)$, we have $\rho_1=\rho_2=\rho_3=\frac{1}{2}$. Hence condition ($\mathfrak{c}$) is not satisfied, and $\W$ is degenerate along $C.$
\smallskip

\noindent Pulling-back $\W$ by the double cover $\pi\hspace{1mm}\colon(x,y)\mapsto(z,w)=(x,y^2)$, an explicit computation gives
\[
K(\pi^*\W)=\left(-\frac{16\hspace{0.2mm}x}{(x^2-1)^2y}+\cdots\right)\mathrm{d}x\wedge\mathrm{d}y.
\]
Thus $K(\pi^*\W)$ is not holomorphic on $\tilde{C}=\{y=0\}$, and therefore $K(\W)$ is not holomorphic on $C=\{w=0\}.$
\end{eg}

%==================================================================================================================================================================

\section{Applications to the flatness of certain webs on $\pp$}\label{sec:platitude-tissus-pp}

\noindent In this section, we study the flatness of products of dual webs of reduced convex foliations on $\pp.$

\subsection{Reminders on the inflection divisor and the dual web of a foliation on $\pp$}\label{subsec:singularité-diviseur-inflexion-tissu-dual}

A degree $d$ holomorphic foliation $\F$ on $\pp$ is defined in homogeneous coordinates $[x:y:z]$ by a $1$-form $$\omega=a(x,y,z)\mathrm{d}x+b(x,y,z)\mathrm{d}y+c(x,y,z)\mathrm{d}z,$$ where $a,$ $b$ and $c$ are homogeneous polynomials of degree $d+1$ without common factor and satisfying the \textsc{Euler} condition $i_{\mathrm{R}}\omega=0$, where $\mathrm{R}=x\frac{\partial{}}{\partial{x}}+y\frac{\partial{}}{\partial{y}}+z\frac{\partial{}}{\partial{z}}$ is the radial vector field and $i_{\mathrm{R}}$ is the interior product by  $\mathrm{R}$. The {\sl singular locus} $\mathrm{Sing}\F$ of $\F$ is the projectivization of the singular locus of~$\omega$ $$\mathrm{Sing}\,\omega=\{(x,y,z)\in\mathbb{C}^3\,\vert \, a(x,y,z)=b(x,y,z)=c(x,y,z)=0\}.$$

\noindent Let $\mathrm{Z}=E\frac{\partial}{\partial x}+F\frac{\partial}{\partial y}+G\frac{\partial}{\partial z}$ be a homogeneous vector field of degree $d$ on  $\mathbb{C}^3$ describing $\F,$ {\it i.e.} such that $\omega=i_{\mathrm{R}}i_{\mathrm{Z}}\mathrm{d}x\wedge\mathrm{d}y\wedge\mathrm{d}z.$ The {\sl inflection divisor} of $\F$, denoted by $\IF$, is the divisor on $\pp$ defined by the equation
\begin{equation*}
\left| \begin{array}{ccc}
x &  E &  \mathrm{Z}(E) \\
y &  F &  \mathrm{Z}(F)  \\
z &  G &  \mathrm{Z}(G)
\end{array} \right|=0.
\end{equation*}

\noindent On $\pp\smallsetminus\Sing\F,$ the support of $\IF$ coincides with the set of inflection points of the leaves of $\F.$ Moreover, $\deg\IF=3d$, \emph{see}~\cite{Per01}.

\noindent The foliation $\F$ is said to be {\sl convex} if the support of $\IF$ is a union of invariant lines. We say that $\F$ is {\sl reduced convex} if it is convex with reduced inflection divisor, that is, if $\IF$ consists of $3d$ distinct invariant lines.
\smallskip

\noindent Let $\ell$ be a generic line in $\pp$ and let $\check{\ell}\in\pd$ be its dual point. The foliation $\F$ is tangent to $\ell$ at exactly $d$ points $p_1,\ldots,p_d.$ The $d$ lines $\check p_i\subset\pd$ dual to the points $p_i$ pass through $\check{\ell}$ and determine the directions at $\check{\ell}$ of a $d$-web on $\pd$, called the {\sl \textsc{Legendre} transform} of $\F$ and denoted by $\Leg\F$. More explicitly, if $\F$ is given in an affine chart $(x,y)$ of $\pp$ by a $1$-form $A(x,y)\mathrm{d}x+B(x,y)\mathrm{d}y$, where $A,B\in\mathbb{C}[x,y]$ and $\gcd(A,B)=1$, then, in the affine chart $(p,q)$ of~$\pd$ associated to the line $\{y=px-q\}\subset\pp,$ $\Leg\F$ is described by the implicit differential equation
\[
\check{F}(p,q,x):=A(x,px-q)+pB(x,px-q)=0, \qquad \text{with} \qquad x=\frac{\mathrm{d}q}{\mathrm{d}p}.
\]

\noindent For $s\in\Sing\F$, we denote by $\tau(\F,s)$ the {\sl tangency order} of $\F$ with a generic line passing through $s$
\begin{align*}
\tau(\F,s)=\min\{\Tang(\F,\ell,s)\hspace{1mm}\vert\hspace{1mm}\ell\in\mathfrak{L}_s(\F)\},
\end{align*}
where $\mathfrak{L}_s(\F)$ stands for the family of lines through $s$ that are not invariant by $\F.$

\noindent When $\F$ is reduced convex, all its singularities are non-degenerate (\cite[Lemma~6.8]{BM18Bull}); by~\cite[Lemma~2.2]{BFM14}, we~have
\begin{equation}\label{equa:Delta-LegF-convexe-reduit}
\Delta(\Leg\F)=\radFd,
\end{equation}
where $\radFd$ denotes the union of the lines dual to the points of $\radF=\{s\in\Sing\F\hspace{1mm}\colon\tau(\F,s)\geq2\}.$ The set $\radF$ consists of the radial singularities of $\F$; the integer $\tau(\F,s)-1$ is the radiality order of $s\in\radF$, \emph{see}~\cite[\S1.3]{BM18Bull}.

\subsection{Flatness of products of dual webs of reduced convex foliations}

Our aim here is to establish the following theorem.

\begin{thm}\label{thm:Leg-ell1-ellr-F1-Fn}
{\sl Let $\F_1,\ldots,\F_n$ be reduced convex foliations on $\pp.$ Assume that the following conditions are satisfied:}
\begin{itemize}
\item [(\texttt{H1})] {\sl For all $1\leq i<j\leq n$, $\Tang(\F_i,\F_j)$ is a union of invariant lines by $\F_i$ and $\F_j.$}
\smallskip

\item [(\texttt{H2})] {\sl For all $1\leq i<j<k\leq n$, there is no common singularity $s\in\Sing\F_i\cap\Sing\F_j\cap\Sing\F_k$ such that $\tau(\F_i,s)=\tau(\F_j,s)=\tau(\F_k,s).$}
\end{itemize}
\smallskip

\noindent {\sl Then the web $\W=\Leg\F_1\boxtimes\cdots\boxtimes\Leg\F_n$ is flat.}
\end{thm}

\begin{proof}[\sl Proof]
For $n=1$, the statement reduces to Theorem~2 of~\cite{MP13}. Assume therefore that $n\geq2$. We have
\begin{align*}
\Delta(\W)&=\bigcup_{i=1}^{n}\Delta(\Leg\F_i)\,\,\cup\,\bigcup_{1\le i<j\le n}\Tang(\Leg\F_i,\Leg\F_j)\\
&=\bigcup_{1\le i<j\le n}\Delta\big(\Leg\F_i\boxtimes\Leg\F_j\big).
\end{align*}

\noindent According to~\cite[Proposition~54]{Pra25}, hypothesis~(\texttt{H1}) and the reduced convexity of $\F_i$ and $\F_j$ imply
\begin{align*}
\Delta\big(\Leg\F_i\boxtimes\Leg\F_j\big)
=\check{\Sigma}_{\F_i}^{\mathrm{rad}}\cup\check{\Sigma}_{\F_j}^{\mathrm{rad}}\cup\check{\Sigma}_{\F_i,\F_j},
\end{align*}
where $\check{\Sigma}_{\F_i,\F_j}$ denotes the union of lines dual to the singularities common to $\F_i$ and $\F_j$. Thus
\begin{align}
\Delta(\W)&=\bigcup_{i=1}^{n}\check{\Sigma}_{\F_i}^{\mathrm{rad}}\,\,\cup\,\bigcup_{1\le i<j\le n}\check{\Sigma}_{\F_i,\F_j}\\
\label{equa:Delta-Leg-F1-Leg-Fn-subset-duales-Sing-Fi}
&\subset\bigcup_{i=1}^{n}\check{\widehat{\Sing\F_i}}.
\end{align}

\noindent Let $s$ be a point of $\cup_{i=1}^{n}\Sing\F_i$ such that its dual line $\check{s}$ is contained in $\Delta(\W)$. Set $d_i:=\deg\F_i$, $I(s):=\left\{i\in\{1,\ldots,n\}\hspace{0.8mm}\vert\hspace{0.8mm}s\in\Sing\F_i\right\}$ and, for $i\in I(s)$, $\tau_i:=\tau(\F_i,s).$ By \cite[Proposition~3.3~and~\S4.2]{MP13}, for every $i\in I(s)$, the $d_i$-web $\Leg\F_i$ decomposes, in a neighborhood of a generic point of $\check{s},$ as $\Leg\F_i=\W_{\tau_i}\boxtimes\W_{d_i-\tau_i},$ where $\W_{d_i-\tau_i}$ is a regular $(d_i-\tau_i)$-web transverse to $\check{s}$, and $\W_{\tau_i}$ is an irreducible $\tau_i$-web leaving $\check{s}$ invariant and satisfying, for $\tau_i\geq2$, $\mathrm{mult}\left(\Delta(\W_{\tau_i}),\check{s}\right)=\tau_i-1.$
\smallskip

\noindent Now, set $\W_0:=\boxtimes_{k\in I^{\mathrm{c}}(s)}\Leg\F_k$, where $I^{\mathrm{c}}(s):=\{1,\ldots,n\}\setminus I(s).$ Since $s\notin\Sing\F_k$ for every $k\in I^{\mathrm{c}}(s)$, each web $\Leg\F_k$ is transverse to $\check{s}.$ Furthermore, applying~(\ref{equa:Delta-Leg-F1-Leg-Fn-subset-duales-Sing-Fi}) to $\W_0$ yields $\Delta(\W_0)\subset\bigcup_{k\in I^{\mathrm{c}}(s)}\check{\widehat{\Sing\F_k}}$, hence $\check{s}\not\subset\Delta(\W_0).$ Thus, $\W_0$ is regular and transverse to $\check{s}.$
\smallskip

\noindent Moreover, an argument in the proof of~\cite[Proposition~55]{Pra25}, using hypothesis~(\texttt{H1}), shows that for all $i,j\in I(s)$ with $i\neq j$ and every $k\in I^{\mathrm{c}}(s),$ we have $\check{s}\not\subset\Tang(\W_{d_i-\tau_i},\W_{d_j-\tau_j})\cup\Tang(\W_{d_i-\tau_i},\Leg\F_k).$
\smallskip

\noindent It follows that the web $\W^{\mathrm{tr}}:=\boxtimes_{i\in I(s)}\W_{d_i-\tau_i}\boxtimes\W_0$ is regular and transverse to $\check{s},$ and we have $\W=\W^{\mathrm{inv}}\boxtimes\W^{\mathrm{tr}}$, where $\W^{\mathrm{inv}}:=\boxtimes_{i\in I(s)}\W_{\tau_i}.$
\smallskip

\noindent Set $\W_{\tau\geq2}:=\boxtimes_{\substack{i\in I(s)\\ \tau_i\geq2}}\W_{\tau_i}.$ From the above discussion and hypothesis~(\texttt{H2}), the web $\W_{\tau\geq2}$ is of the form of~the~web in Example~\ref{eg:tissu-non-degenere} and is therefore non-degenerate along $\check{s}$. Again, hypothesis~(\texttt{H2}) implies that $\#\,\{i\in I(s)\hspace{0.8mm}\vert\hspace{0.8mm}\tau_i=1\}\leq2.$ As a result, denoting $\mathscr{F}_{i}:=\W_{\tau_i}$ for every $i\in I(s)$ such that $\tau_i=1,$ $\W^{\mathrm{inv}}$ is of one of the following three forms:
\smallskip

\begin{itemize}
\item [(\texttt{i})] $\W^{\mathrm{inv}}=\W_{\tau\geq2}$;
\smallskip
\item [(\texttt{ii})] $\W^{\mathrm{inv}}=\mathscr{F}_{i_0}\boxtimes\W_{\tau\geq2}$;
\smallskip
\item [(\texttt{iii})] $\W^{\mathrm{inv}}=\mathscr{F}_{i_0}\boxtimes\mathscr{F}_{j_0}\boxtimes\W_{\tau\geq2}$, with $i_0\neq j_0.$
\end{itemize}
\smallskip

\noindent In each of these cases, the curvature of $\W$ is holomorphic along $\check{s}$ by applying Theorem~\ref{thm:W-n-fort} or Remark~\ref{rem:W-n-W-n-nu-W-nu}. The component $\check{s}\subset\Delta(\W)$ being arbitrary, $K(\W)$ is holomorphic on $\pd$. Consequently, $K(\W)\equiv0.$
\end{proof}

\medskip

\noindent For $n=2$, we immediately obtain the following corollary.
\begin{cor}\label{cor:Leg-F1-Leg-F2}
{\sl Let $\F_1$ and $\F_2$ be two reduced convex foliations on $\pp$ of respective degrees $d_1,d_2\geq2$. Assume that $\Tang(\F_1,\F_2)$ is a union of invariant lines by $\F_1$ and $\F_2$. Then the $(d_1+d_2)$-web $\W=\Leg\F_1\boxtimes\Leg\F_2$ is flat.
}
\end{cor}

\noindent This corollary was already formulated by \textsc{Pracias} and \textsc{Falla Luza} in~\cite{PF24}. To illustrate it, let us consider the following example due to them.

\begin{eg}
The \textsc{Fermat} foliation $\F_{0}^{d}$ of degree $d$ is defined in homogeneous coordinates by the $1$-form
\begin{align*}
\Omega_{0}^{d}=x^{d}(y\mathrm{d}z-z\mathrm{d}y)+y^{d}(z\mathrm{d}x-x\mathrm{d}z)+z^{d}(x\mathrm{d}y-y\mathrm{d}x).
\end{align*}
According to~\cite[\S5]{MP13}, $\F_{0}^{d}$ is a reduced convex foliation whose $3d$ invariant lines are
\begin{align*}
&
x=0,\quad
y=0,\quad
z=0,\quad
y=\zeta\,x,\quad
y=\zeta\,z,\quad
x=\zeta\,z,\quad
\text{with}\hspace{1mm}\zeta^{d-1}=1.
\end{align*}
Following~\cite{PF24}, taking two \textsc{Fermat} foliations $\F_{0}^{n}$ and $\F_{0}^{d}$ with $d=2n-1$, we have $\mathrm{I}_{\F_{0}^{n}}\subset\mathrm{I}_{\F_{0}^{d}}$ and~$\Tang(\F_{0}^{n},\F_{0}^{d})=\mathrm{I}_{\F_{0}^{n}}.$ In other words, $\Tang(\F_{0}^{n},\F_{0}^{d})$ consists of the $3n$ invariant lines of $\F_{0}^{n}$, which are also invariant by $\F_{0}^{d}.$ Corollary~\ref{cor:Leg-F1-Leg-F2} then shows that the web $\W=\Leg\F_{0}^{n}\boxtimes\Leg\F_{0}^{d}$ is flat.
\end{eg}
\smallskip

\noindent Let us give another example of application of Corollary~\ref{cor:Leg-F1-Leg-F2}.

\begin{eg}
Consider the \textsc{Hesse} foliations $\Hesse$ and $\Hessesept$ of respective degrees $4$ and $7$, defined in homogeneous coordinates by the $1$-forms
\smallskip
\begin{align*}
&\Omegahesse=yz(2\,x^3-y^3-z^3)\mathrm{d}x+xz(2y^3-x^3-z^3)\mathrm{d}y+xy(2z^3-x^3-y^3)\mathrm{d}z,\\
\smallskip
&\Omegahessesept=yz(y^3-z^3)(7x^3+y^3+z^3)\mathrm{d}x+xz(z^3-x^3)(x^3+7y^3+z^3)\mathrm{d}y+xy(x^3-y^3)(x^3+y^3+7z^3)\mathrm{d}z.
\end{align*}
They are reduced convex foliations (\emph{see}~\cite[\S5]{MP13}) and, by~\cite[\S8]{Bed24arxiv2}, we have $\mathrm{I}_{\Hesse}\subset\mathrm{I}_{\Hessesept}$, hence $\mathrm{I}_{\Hesse}\subset\Tang(\Hesse,\Hessesept).$
Now,
\begin{small}
$\deg\Tang(\Hesse,\Hessesept)=\deg\Hesse+\deg\Hessesept+1=12=\deg\mathrm{I}_{\Hesse}.$
\end{small} Therefore, $\Tang(\Hesse,\Hessesept)=\mathrm{I}_{\Hesse}$, {\it i.e.} $\Tang(\Hesse,\Hessesept)$ consists of the $12$ invariant lines of $\Hesse$, which are also invariant by $\Hessesept$. By Corollary~\ref{cor:Leg-F1-Leg-F2}, the web $\W=\Leg\Hesse\boxtimes\Leg\Hessesept$ is flat.
\end{eg}

\noindent The following example illustrates Theorem~\ref{thm:Leg-ell1-ellr-F1-Fn} for $n=3.$

\begin{eg}
We define three foliations $\F_i$ ($i=1,2,3$) conjugate to the \textsc{Fermat} foliation $\F_0^2$ of degree~$2$ by setting
\[
\F_i=\varphi_i^*\F_0^2,\qquad
\text{where}\hspace{2mm}
\varphi_1=[x:z-y:z],\quad
\varphi_2=[x+y:y:z],\quad
\varphi_3=[x:x+y:z].
\]
They are therefore reduced convex foliations, given by the $1$-forms $\varphi_i^*\Omega_0^2.$ One readily verifies that the triple $(\F_1,\F_2,\F_3)$ satisfies hypotheses \texttt{(H1)} and \texttt{(H2)}, so that the $6$-web $\W=\Leg\F_1\boxtimes\Leg\F_2\boxtimes\Leg\F_3$ is flat (Theorem~\ref{thm:Leg-ell1-ellr-F1-Fn}).
\end{eg}

\noindent Finally, let us present an example showing that hypothesis~\texttt{(H2)} cannot be removed from Theorem~\ref{thm:Leg-ell1-ellr-F1-Fn}.

\begin{eg}\label{eg:H2-necessaire}
Consider the foliations $\F_i=\varphi_i^*\F_0^2$ ($i=1,2,3$), where
\[
\varphi_1=[z-x:y:z],\qquad
\varphi_2=[z-x:-y:z],\qquad
\varphi_3=[y+z-x:y:z].
\]
They are respectively described in the affine chart $z=1$ by the $1$-forms
\begin{align*}
&\omega_1=y\mathrm{d}x+x\mathrm{d}y-y^2\mathrm{d}x-x^2\mathrm{d}y,\\
&\omega_2=y\mathrm{d}x+x\mathrm{d}y+y^2\mathrm{d}x-x^2\mathrm{d}y,\\
&\omega_3=y\mathrm{d}x+(x-2y)\mathrm{d}y-y^2\mathrm{d}x-x(x-2y)\mathrm{d}y.
\end{align*}

\noindent One easily checks that the triple $(\F_1,\F_2,\F_3)$ satisfies \texttt{(H1)}. However, it does not satisfy \texttt{(H2)}, since the point $s=[0:0:1]$ is a common singularity of all the $\F_i$ and is non-radial for each of them.
\smallskip

\noindent Let $\W=\Leg\F_1\boxtimes\Leg\F_2\boxtimes\Leg\F_3.$ We will show that the curvature of $\W$ is not holomorphic along the line dual to $s$, which will imply that $\W$ is not flat. In the affine chart $(p,q)$ of~$\pd$ associated to the line $\{y=px-q\}\subset\pp,$ the $2$-webs $\Leg\F_1$, $\Leg\F_2$ and $\Leg\F_3$ are respectively given by the $2$-symmetric $1$-forms
\begin{align*}
&\check{\omega}_1=p(p+1)\mathrm{d}q^2-2p(q+1)\mathrm{d}p\mathrm{d}q+q(q+1)\mathrm{d}p^2,\\
&\check{\omega}_2=p(p-1)\mathrm{d}q^2-2p(q-1)\mathrm{d}p\mathrm{d}q+q(q-1)\mathrm{d}p^2,\\
&\check{\omega}_3=p(p-1)\mathrm{d}q^2-2p(p-1)\mathrm{d}p\mathrm{d}q+q(2p-q-1)\mathrm{d}p^2.
\end{align*}

\noindent In a neighborhood of a generic point of $\check{s}=\{q=0\}$, the $6$-web $\W$ decomposes as $\W=\W^\mathrm{inv}\boxtimes\W^\mathrm{tr}$, where $\W^\mathrm{tr}$ is a regular $3$-web transverse to $\check{s}$, and $\W^\mathrm{inv}=\mathscr{F}_1\boxtimes\mathscr{F}_2\boxtimes\mathscr{F}_3$ consists of foliations leaving $\check{s}$ invariant, with
\begin{small}
\begin{align*}
&\mathscr{F}_1:\mathrm{d}q+q\left(-\frac{1}{2p}-\frac{p+1}{8p^2}q+\cdots\right)\mathrm{d}p=0,\\
&\mathscr{F}_2:\mathrm{d}q+q\left(-\frac{1}{2p}+\frac{p-1}{8p^2}q+\cdots\right)\mathrm{d}p=0,\\
&\mathscr{F}_3:\mathrm{d}q+q\left(-\frac{1}{2p}-\frac{1}{2(p-1)}-\frac{1}{8p^2(p-1)^2}q+\cdots\right)\mathrm{d}p=0.
\end{align*}
\end{small}
\hspace{-1mm}A direct computation, using~\cite[formula~(1.1)]{BM18Bull}, shows that $\eta(\W^\mathrm{inv})=\left(\frac{4p+1}{pq}+\cdots\right)\mathrm{d}q,$ so that
\[
K(\W^\mathrm{inv})=\mathrm{d}\eta(\W^\mathrm{inv})=\left(-\frac{1}{p^2q}+\cdots\right)\mathrm{d}p\wedge\mathrm{d}q.
\]
In particular, $K(\W^\mathrm{inv})$ is not holomorphic along $\check{s}=\{q=0\}$, and hence neither is $K(\W)$ by~Theorem~\ref{thm:W-n-W-d-n}.
\end{eg}

%==================================================================================================================================================================

\end{document}